\documentclass[12pt]{article}
\usepackage{makeidx}
\usepackage{amsfonts}
\usepackage{amssymb}
\usepackage{amsmath}
\usepackage[numbers,sort&compress]{natbib}
\usepackage[colorlinks,citecolor=blue,urlcolor=blue]{hyperref}
\usepackage{graphicx,indentfirst,tabularx}

\newtheorem{theorem}{Theorem}

\newtheorem{conjecture}[theorem]{Conjecture}

\newtheorem{example}[theorem]{Example}

\newtheorem{remark}[theorem]{Remark}

\newenvironment{proof}[1][Proof]{\noindent\textbf{#1.} }{\ \rule{0.5em}{0.5em}}
\input{tcilatex}
\begin{document}

\title{The Wronski orthogonalization process in Hilbert function spaces}
\author{Athanasios C. Micheas \thanks{%
Department of Statistics, University of Missouri, 146 Middlebush Hall,
Columbia, MO 65211-6100, USA, email: micheasa@missouri.edu}}
\maketitle

\begin{abstract}
Based on the Wronski determinant, we propose the construction of linearly
independent and orthogonal functions in any Hilbert function space. We call
the new method the Wronski orthogonalization process. The method requires
only an initial function from the space of functions under consideration,
that satisfies mild conditions, and emerges as a generalization of the
Gram-Schmidt process. Two applications are considered, including solutions
to ordinary differential equations and the construction of basis functions
in Hilbert function spaces. We also present a conjecture that connects the
latter two concepts, which leads to the introduction of the Wronski basis.
\end{abstract}



\section{Introduction}

Linearly independent and orthogonal sets of functions constitute a
foundation of modern mathematics and its applications, most notably in
functional and numerical analysis. Such sets are used to create new
methodologies with important applications in many different scientific
disciplines, including chemistry, engineering, physics and statistics.

While classic methods, such as the Gram-Schmidt orthonormalization, allow us
to build orthonormal bases in function spaces, they always require a
starting collection of linearly independent functions from the function
space (see \cite{Dudley2004}, Theorem 5.4.6, \cite{ovchinnikov2018functional}%
, Theorem 7.14). Moreover, finding the totality of solutions to
inhomogeneous, linear, ordinary differential equations (ODE) requires a set
of fundamental solutions (linearly independent solutions) of the homogeneous
ODE (see \cite{barbu2016differential}, \cite{zwillinger2021handbook} and
Remark \ref{variationofparameters} below). In this paper, we propose methods
that address these important problems in mathematics.

In what follows, we consider a general function space%
\begin{equation}
\mathcal{F}_{\mathbb{M}}^{\mathbb{M}}=\{f:\mathbb{M}\rightarrow \mathbb{M},%
\text{ }f\text{ measurable in }(\mathbb{M},\mathcal{B}(\mathbb{M}))\},
\label{FunctionSpace1}
\end{equation}%
with $\mathcal{B}(\mathbb{M})$ the Borel sets of $\mathbb{M},$ where we use $%
\mathbb{M}$ to denote the real numbers $\Re $ or the complex plane $\mathbb{C%
}$. Recall that a set of functions $\mathcal{F}_{N}=\{f_{1},...,f_{N}\}%
\subset \mathcal{F}_{\mathbb{M}}^{\mathbb{M}},$ forms a set of linearly
independent functions if the only set of constants $\mathbf{c}%
=[c_{1},...,c_{N}]\in \mathbb{M}^{N},$ which allows us to write%
\begin{equation}
\sum_{i=1}^{N}c_{i}f_{i}(x)=0,  \label{LinIndep}
\end{equation}%
is the zero vector, i.e., $c_{i}=0$, for all $i=1,2..,N$. In other words, no
function in the set $\mathcal{F}_{N}$ can be written as a linear combination
of other functions in $\mathcal{F}_{N}$.

The classic method of verifying Equation (\ref{LinIndep}) is via the Wronski
determinant (see \cite{zwillinger2021handbook}, pg 13, \cite%
{bostan2010wronskians}, \cite{purbhoo2013wronskians}), denoted by $%
W(f_{1},...,f_{N}),$ which is defined as the determinant%
\begin{equation}
W(f_{1},...,f_{N})=\left\vert
\begin{tabular}{lll}
$f_{1}(x)$ & $...$ & $f_{N}(x)$ \\
$\frac{df_{1}(x)}{dx}$ & $...$ & $\frac{df_{N}(x)}{dx}$ \\
$...$ &  & $...$ \\
$\frac{d^{N-1}f_{1}(x)}{dx^{N-1}}$ & $...$ & $\frac{d^{N-1}f_{N}(x)}{dx^{N-1}%
}$%
\end{tabular}%
\right\vert ,  \label{Wronskidet}
\end{equation}%
for any $x\in \mathbb{M},$ provided that all derivatives up to the order $%
N-1 $ of $f\in \mathcal{F}_{N},$ exist. Then $\mathcal{F}_{N}$ is linearly
independent if $W(f_{1},...,f_{N})\neq 0,$ for all $x\in \mathbb{M}$, and
therefore, the functions in the collection $\mathcal{F}_{N}$ must be up to $%
(N-1)$-order differentiable in order to perform the Wronski criterion. As a
result, we write $\mathcal{F}_{\mathbb{M}}^{N,\mathbb{M}}\subset \mathcal{F}%
_{\mathbb{M}}^{\mathbb{M}}$, for the collection of functions from $\mathcal{F%
}_{\mathbb{M}}^{\mathbb{M}}$ that are continuously, differentiable up to the
$N$th order. For countable sets of functions $\mathcal{F}_{\infty
}=\{f_{1},f_{2},...\}\subset \mathcal{F}_{\mathbb{M}}^{\mathbb{M}}$,
derivatives of all order are required, which can be satisfied by assuming
that any $f\in \mathcal{F}_{\infty },$ is an analytic function. Note that $%
\mathcal{F}_{\infty }$ is linearly independent if and only if any finite
subset is linearly independent.

In order to study the properties of the general function space $\mathcal{F}_{%
\mathbb{M}}^{\mathbb{M}}$ we require concepts such as the length of a
function, a basis that helps describe each element of $\mathcal{F}_{\mathbb{M%
}}^{\mathbb{M}}$, and limits of elements of $\mathcal{F}_{\mathbb{M}}^{%
\mathbb{M}}$, and so forth. A Hilbert function space (see \cite%
{isaev2021necessary}) provides all the important ingredients we need in
order to perform all of the aforementioned tasks, and more. Therefore, we
will assume that the function space $\mathcal{F}_{\mathbb{M}}^{\mathbb{M}}$
under consideration is Hilbert based on the inner product $\rho .$ Then, we
can immediately equip $\mathcal{F}_{\mathbb{M}}^{\mathbb{M}}$ with the
induced norm%
\begin{equation}
\left\Vert f\right\Vert _{\rho }=\sqrt{\rho (f,f)},  \label{MTMNorm}
\end{equation}%
and distance%
\begin{equation}
d\left( f_{1},f_{2}\right) =\left\Vert f_{1}-f_{2}\right\Vert _{\rho }=\sqrt{%
\rho \left( f_{1}-f_{2},f_{1}-f_{2}\right) },  \label{MTMDistance}
\end{equation}%
such that $\mathcal{F}_{\mathbb{M}}^{\mathbb{M}}$ becomes a normed vector
(linear) space and $\left( \mathcal{F}_{\mathbb{M}}^{\mathbb{M}},d\right) $
is a metric space.

In this paper, we propose an orthogonalization construction for Hilbert
function spaces, in order to:

\begin{itemize}
\item simplify, streamline and generalize the Gram-Schmidt
orthonormalization process. In particular, we no longer require a collection
of $N$ linearly independent functions beforehand, as in the Gram-Schmidt
process, in order to construct an orthonormal set of functions or function
basis;

\item provide a subset of linearly independent and orthogonal functions by
construction in any Hilbert function space. Importantly, all that is
required is the knowledge of a single function from the space, that
satisfies mild conditions;

\item create a flexible framework that is highly amenable to changes and
applicable to different contexts in mathematics;

\item give us a new method of creating bases in Hilbert function spaces;

\item provide a set of fundamental solutions for linear, homogeneous, $n$th
order ODEs, and as a result, the totality of solutions for any linear,
inhomogeneous ODE.
\end{itemize}

The new method will be aptly called the Wronski orthogonalization process.
In this introductory paper, the Wronski process will help us answer two
important questions; firstly, can we build a basis in any Hilbert function
space, using only a single function from the space? Second, can we obtain
all solutions of an ODE given a single fundamental solution of the
corresponding homogeneous ODE?

The paper proceeds as follows. In the next section we discuss the underlying
details of the proposed Wronski orthogonalization procedure and show that it
is a generalization of the Gram-Schmidt process. In Section 3 we present an
application of the method to solving linear, inhomogeneous ODEs and
introduce the Wronski basis conjecture. Concluding remarks are given in the
last section.

\section{Wronski Orthogonalization}

Consider $\mathcal{F}_{\mathbb{M}}^{N,\mathbb{M}}$ and let $\rho $ be an
inner product, such that $(\mathcal{F}_{\mathbb{M}}^{N,\mathbb{M}},\rho )$
is a Hilbert space. Let $\mathcal{F}_{N}=\{f_{1},...,f_{N}\}$ $\subset
\mathcal{F}_{\mathbb{M}}^{N,\mathbb{M}}$, and assume that $f_{1}(x)\neq 0,$ $%
\forall x\in \mathbb{M}$, with $0<\left\Vert f_{1}\right\Vert _{\rho
}<+\infty .$ Although any measure can be used, we will use the Lebesgue
measure $\mu $ over $(\mathbb{M},\mathcal{B}(\mathbb{M})),$ as the
integrating measure, in the formulation of any integrals that follow.

The construction we propose can be thought of as the inverse process of the
variation of parameters method for solving linear, inhomogeneous, $n$th
order differential equations. In order to appreciate this idea, the
following remark collects the approach to solving differential equations
based on Wronskians (more details of this classic method can be found in
\cite{zwillinger2021handbook}).

\begin{remark}[Variation of parameters]
\label{variationofparameters}Consider the $n$th order, inhomogeneous, linear
differential equation%
\begin{gather}
a_{n}(x)\frac{d^{n}y(x)}{dx^{n}}+a_{n-1}(x)\frac{d^{n-1}y(x)}{dx^{n-1}}+...
\label{Inhomonorder} \\
+a_{1}(x)\frac{dy(x)}{dx}+a_{0}(x)y(x)=b(x),  \notag
\end{gather}%
where $b,a_{0},...,a_{n}$ are continuous functions over\ a subset $I\subset
\mathbb{M}$ (interval if $\mathbb{M}=\Re $ or rectangle if $\mathbb{M}=%
\mathbb{C}$), with $a_{n}(x)\neq 0,$ for all $x\in I$. Let the homogeneous
version of (\ref{Inhomonorder}),%
\begin{gather}
a_{n}(x)\frac{d^{n}y(x)}{dx^{n}}+a_{n-1}(x)\frac{d^{n-1}y(x)}{dx^{n-1}}+...
\label{homonorder} \\
+a_{1}(x)\frac{dy(x)}{dx}+a_{0}(x)y(x)=0,  \notag
\end{gather}%
has a set of linearly independent solutions $\{y_{1},...,y_{n}\}.$\newline
Now, the variation of parameters method provides the general solution of (%
\ref{Inhomonorder}), given by%
\begin{equation*}
y(x)=y_{0}(x)+\sum\limits_{i=1}^{n}c_{i}y_{i}(x),
\end{equation*}%
$x\in I$, where $y_{0}(x)$ is a specific solution of (\ref{Inhomonorder}),
and $c_{1},c_{2},...,c_{n}$ are real constants. The specific solution can be
written in terms of Wronskians as%
\begin{equation*}
y_{0}(x)=\sum\limits_{k=1}^{n}y_{k}(x)\int\limits_{x_{0}}^{x}\frac{%
W_{k}(y_{1},...,y_{n})(t)}{W(y_{1},...,y_{n})(t)}\frac{b(t)}{a_{n}(t)}d\mu
(t),
\end{equation*}%
for any $x_{0}\in I$, where $W_{k}(y_{1},...,y_{n})$ denotes the Wronskian $%
W(y_{1},...,y_{n}),$ with its $k$-column replaced by $(0,0,...,0,1)$. By
construction, $y_{0}(x)$ satisfies the initial conditions $\frac{%
d^{n-1}y_{0}(x_{0})}{dx^{n-1}}$ $=...=y_{0}(x_{0})=0.$\newline
In addition, recall that if $\{y_{1},...,y_{n}\}$ are $n$-order
differentiable, with $W(y_{1},$ $...,$ $y_{n})$ $\neq 0,$ for all $x\in
\mathbb{M}$, then there exists a unique, homogeneous, linear differential
equation of the form (\ref{Inhomonorder}), with $a_{n}(x)=1$, given by%
\begin{equation}
\frac{W(y_{1},...,y_{n},y)}{W(y_{1},...,y_{n})}=0,
\label{HomogeneousWronski}
\end{equation}%
such that $\{y_{1},...,y_{n}\}$ forms a set of fundamental solutions.
\end{remark}

Consider any $x_{0}\in I\subset \mathbb{M}$. We use the notation
\begin{eqnarray}
F_{h,f_{1},f_{2},...,f_{n}}(x)
&=&\sum\limits_{k=1}^{n}f_{k}(x)\int\limits_{x_{0}}^{x}\frac{%
W_{k}(f_{1},f_{2},...,f_{n})(t)}{W(f_{1},f_{2},...,f_{n})(t)}  \notag \\
&&h(t)d\mu (t),  \label{Ffunc1}
\end{eqnarray}%
where $h$ is a known, measurable weight function, such that the latter
integral exists and satisfies $F_{h,f_{1},f_{2},...,f_{n}}(x)<+\infty ,$
with $F_{h,f_{1},f_{2},...,f_{n}}(x_{0})=$ $\sum%
\limits_{k=1}^{n}f_{k}(x_{0}).$ For $n=1$, under similar assumptions for $h$%
, we write%
\begin{equation}
F_{h,f_{1}}(x)=f_{1}(x)\int\limits_{x_{0}}^{x}\frac{h(t)}{f_{1}(t)}d\mu
(t)<+\infty ,  \label{Ffunc0}
\end{equation}%
with $F_{h,f_{1}}(x_{0})=f_{1}(x_{0}).$

Given the initial function $f_{1}$, we begin the Wronski orthogonalization
construction as follows; any function $f_{2}$ is linearly independent of the
initial $f_{1},$ if the Wronski determinant%
\begin{eqnarray*}
W(f_{1},f_{2})(x) &=&\left\vert
\begin{tabular}{ll}
$f_{1}(x)$ & $f_{2}(x)$ \\
$\frac{df_{1}(x)}{dx}$ & $\frac{df_{2}(x)}{dx}$%
\end{tabular}%
\right\vert \\
&=&f_{1}(x)\frac{df_{2}(x)}{dx}-\frac{df_{1}(x)}{dx}f_{2}(x),
\end{eqnarray*}%
is non-zero for all $x\in \mathbb{M}$. We will construct $f_{2}$ in order to
guarantee $W(f_{1},f_{2})(x)\neq 0,$ for all $x\in \mathbb{M}$.

A crucial point of the construction is the following step, where we have the
option to set the Wronskian $W(f_{1},f_{2})(x)$ to a non-zero function $h$.
Note here that any choice could be made for $h$, which would lead to a new
function $f_{2}$, by solving the differential equation $%
W(f_{1},f_{2})(x)=h(x)\neq 0$, $\forall x\in \mathbb{M}.$ However, we will
choose specifically%
\begin{equation}
W(f_{1},f_{2})(x)=h_{1}(x)f_{1}(x)\neq 0,  \label{Wronski2Det}
\end{equation}%
$\forall x\in \mathbb{M},$ with $h_{1}\neq 0$, which leads to the
inhomogeneous, linear, first order differential equation%
\begin{equation}
\frac{df_{2}(x)}{dx}+\left( -\frac{1}{f_{1}(x)}\frac{df_{1}(x)}{dx}\right)
f_{2}(x)=h_{1}(x),  \label{DiffEqWronski}
\end{equation}%
with general solution%
\begin{equation}
f_{2}(x)=f_{1}(x)\left[ c+\int\limits_{x_{0}}^{x}\frac{h_{1}(t)}{f_{1}(t)}%
d\mu (t)\right] ,  \label{LinIndepFunc2}
\end{equation}%
where $c\in \mathbb{M}$, an arbitrary constant. Consequently, by
construction, the set of functions $\mathcal{F}_{2}=\{f_{1},f_{2}\}$ is
linearly independent, and%
\begin{equation*}
f_{2}(x)=cf_{1}(x)+F_{h_{1},f_{1}}(x).
\end{equation*}

In order to make $\mathcal{F}_{2}$ an orthogonal set, we look for a specific
constant $c\in \mathbb{M}$ such that $f_{1}$ and $f_{2}$ are orthogonal with
respect to $\rho $, i.e., we must have $\rho \left( f_{1},f_{2}\right) =0.$
Using properties of the inner product $\rho $ we can write%
\begin{eqnarray*}
0 &=&\rho \left( f_{1},f_{2}\right) =\rho \left(
f_{1},cf_{1}+F_{h_{1},f_{1}}\right) \\
&=&\rho \left( f_{1},cf_{1})+\rho (f_{1},F_{h_{1},f_{1}}\right) \\
&=&c\rho (f_{1},f_{1})+\rho (f_{1},F_{h_{1},f_{1}}),
\end{eqnarray*}%
so that solving for $c$ above we have%
\begin{equation*}
c=-\frac{\rho (f_{1},F_{h_{1},f_{1}})}{\left\Vert f_{1}\right\Vert _{\rho
}^{2}}.
\end{equation*}%
Therefore, the linearly independent and orthogonal function to $f_{1},$ is
written as%
\begin{equation*}
f_{2}(x)=F_{h_{1},f_{1}}(x)-\frac{f_{1}(x)}{\left\Vert f_{1}\right\Vert
_{\rho }^{2}}\rho (f_{1},F_{h_{1},f_{1}}),
\end{equation*}%
which is, not surprisingly, a similar form of the orthonormal elements of
the Gram-Schmidt orthonormalization theorem (see \cite{Dudley2004}, Theorem
5.4.6). Note in particular, that we could start the process with a function
of unit length $f_{1}^{0}=\frac{f_{1}}{\left\Vert f_{1}\right\Vert _{\rho }}%
, $ and then set $f_{2}^{0}=\frac{f_{2}}{\left\Vert f_{2}\right\Vert _{\rho }%
}, $ such that $\mathcal{F}_{2}$ becomes an orthonormal set, but we will
present the construction for the orthogonal case. Although it might not be
obvious at this stage of the construction, the Wronski orthogonalization can
be thought of as a generalization of the Gram-Schmidt process, where we do
not require linear independence beforehand; instead, linear independence is
embedded in the new function $f_{2}$ by construction. We exemplify this
relationship in Remark \ref{relationGS}\ below.

Proceed in a similar manner; assume that $\{f_{1},f_{2}\}$ are linearly
independent and orthogonal, and consider the Wronskian given by%
\begin{gather*}
W(f_{1},f_{2},f_{3})(x)=\left\vert
\begin{tabular}{lll}
$f_{1}(x)$ & $\frac{df_{1}(x)}{dx}$ & $\frac{d^{2}f_{1}(x)}{dx^{2}}$ \\
$f_{2}(x)$ & $\frac{df_{2}(x)}{dx}$ & $\frac{d^{2}f_{2}(x)}{dx^{2}}$ \\
$f_{3}(x)$ & $\frac{df_{3}(x)}{dx}$ & $\frac{d^{2}f_{3}(x)}{dx^{2}}$%
\end{tabular}%
\right\vert \\
=W_{3}(f_{1},f_{2},f_{3})\frac{d^{2}f_{3}(x)}{dx^{2}}%
-W_{2}(f_{1},f_{2},f_{3})\frac{df_{3}(x)}{dx} \\
+W_{1}(f_{1},f_{2},f_{3})f_{3}(x),
\end{gather*}%
and setting%
\begin{eqnarray*}
W(f_{1},f_{2},f_{3})(x) &=&h_{2}(x)W_{3}(f_{1},f_{2},f_{3}) \\
&=&h_{2}(x)W(f_{1},f_{2})(x)\neq 0,
\end{eqnarray*}%
for some function $h_{2}\neq 0,$ we have the inhomogeneous, linear, second
order differential equation%
\begin{gather}
\frac{d^{2}f_{3}(x)}{dx^{2}}-\frac{W_{2}(f_{1},f_{2},f_{3})}{%
W_{3}(f_{1},f_{2},f_{3})}\frac{df_{3}(x)}{dx}  \notag \\
+\frac{W_{1}(f_{1},f_{2},f_{3})}{W_{3}(f_{1},f_{2},f_{3})}f_{3}(x)=h_{2}(x).
\label{Thirdorder}
\end{gather}%
The homogeneous version of (\ref{Thirdorder}) is given by $\frac{%
W(f_{1},f_{2},y)}{W(f_{1},f_{2})}=0,$ with all its solutions given by $%
y=c_{1}f_{1}+c_{2}f_{2},$ for some constants $c_{1},c_{2}\in \mathbb{M}$.
Then the totality of solutions of (\ref{Thirdorder}) are given by
\begin{eqnarray*}
f_{3}(x) &=&c_{1}f_{1}(x)+c_{2}f_{2}(x) \\
&&+\sum\limits_{k=1}^{2}f_{k}(x)\int\limits_{x_{0}}^{x}\frac{%
W_{k}(f_{1},f_{2})(t)}{W(f_{1},f_{2})(t)}h_{2}(t)d\mu (t) \\
&=&c_{1}f_{1}(x)+c_{2}f_{2}(x)+F_{h_{2},f_{1},f_{2}}(x),
\end{eqnarray*}%
and as a result, we will look for constants $c_{1},c_{2}$ such that%
\begin{equation*}
\rho \left( f_{1},f_{3}\right) =0,\text{ and }\rho \left( f_{2},f_{3}\right)
=0.
\end{equation*}%
Since $\rho (f_{1},f_{2})=0,$ we can write%
\begin{eqnarray*}
0 &=&\rho \left( f_{1},f_{3}\right) =\rho \left(
f_{1},c_{1}f_{1}+c_{2}f_{2}+F_{h_{2},f_{1},f_{2}}\right) \\
&=&c_{1}\rho (f_{1},f_{1})+c_{2}\rho (f_{1},f_{2})+\rho \left(
f_{1},F_{h_{2},f_{1},f_{2}}\right) \\
&=&c_{1}\left\Vert f_{1}\right\Vert _{\rho }^{2}+\rho \left(
f_{1},F_{h_{2},f_{1},f_{2}}\right) ,
\end{eqnarray*}%
and therefore%
\begin{equation*}
c_{1}=-\frac{1}{\left\Vert f_{1}\right\Vert _{\rho }^{2}}\rho \left(
f_{1},F_{h_{2},f_{1},f_{2}}\right) .
\end{equation*}%
Similarly, we have%
\begin{equation*}
c_{2}=-\frac{1}{\left\Vert f_{2}\right\Vert _{\rho }^{2}}\rho \left(
f_{2},F_{h_{2},f_{1},f_{2}}\right) ,
\end{equation*}%
and as a result, the general solution of (\ref{Thirdorder}) is given by%
\begin{eqnarray*}
f_{3}(x) &=&-\frac{f_{1}(x)}{\left\Vert f_{1}\right\Vert _{\rho }^{2}}\rho
\left( f_{1},F_{h_{2},f_{1},f_{2}}\right) \\
&&-\frac{f_{2}(x)}{\left\Vert f_{2}\right\Vert _{\rho }^{2}}\rho \left(
f_{2},F_{h_{2},f_{1},f_{2}}\right) +F_{h_{2},f_{1},f_{2}}(x),
\end{eqnarray*}%
which makes the set $\{f_{1},f_{2},f_{3}\}$ linearly independent and
orthogonal. Continuing this process we can easily obtain $N$ linearly
independent and orthogonal functions $\mathcal{F}_{N}=\{f_{1},...,f_{N}\},$
all of which are created based on linear combinations of $f_{1}\in \mathcal{F%
}_{\mathbb{M}}^{N,\mathbb{M}},$ and therefore, $\mathcal{F}_{N}\subset
\mathcal{F}_{\mathbb{M}}^{N,\mathbb{M}},$ by construction.

Owing to the exposition above, we are now ready to present the main result
of the general form of the Wronski orthogonalization process.

\begin{theorem}[Wronski Orthogonalization Process]
\label{thmwronski}Assume that $\rho $ is an inner product and $(\mathcal{F}_{%
\mathbb{M}}^{N,\mathbb{M}},\rho ),$ $N<+\infty ,$ is a Hilbert function
space. Furthermore, assume that there exists a non-zero function $f_{1}\in
\mathcal{F}_{\mathbb{M}}^{N,\mathbb{M}},$ with $0<\left\Vert
f_{1}\right\Vert _{\rho }<+\infty .$ We can then create a set of linearly
independent and mutually orthogonal functions with respect to $\rho ,$ given
by $\mathcal{F}_{N}=\{f_{1},...,f_{N}\}$ $\subset \mathcal{F}_{\mathbb{M}%
}^{N,\mathbb{M}}$, where%
\begin{eqnarray}
f_{k}(x) &=&F_{h_{k-1},f_{1},f_{2},...,f_{k-1}}(x)  \label{sol1} \\
&&-\sum\limits_{i=1}^{k-1}\frac{f_{i}(x)}{\left\Vert f_{i}\right\Vert _{\rho
}^{2}}\rho \left( f_{i},F_{h_{k-1},f_{1},f_{2},...,f_{k-1}}\right) ,  \notag
\end{eqnarray}%
for $k=2,...,N,$ with $F_{h_{n},f_{1},f_{2},...,f_{n}}$ given by Equation (%
\ref{Ffunc1}), where the weight functions $h_{n}$ are chosen arbitrarily,
albeit, they are required to have no zeroes and be such that $%
F_{h_{n},f_{1},f_{2},...,f_{n}}(x)<+\infty ,$ for all $n=1,2,...,N-1.$%
\newline
In addition, $f_{k}(x)$ is obtained, by construction, as a solution to the
inhomogeneous, linear, $(k-1)$th order differential equation%
\begin{gather}
\sum\limits_{i=0}^{k-1}(-1)^{i}W_{i+1}(f_{1},f_{2},...,f_{k})\frac{%
d^{i}f_{k}(x)}{dx^{i}}=  \label{difeq1} \\
h_{k-1}(x)W_{k}(f_{1},f_{2},...,f_{k}),  \notag
\end{gather}%
for all $k=2,...,N.$
\end{theorem}

\begin{proof}
The proof is straightforward by induction on $N$.
\end{proof}

The following remark summarizes the benefits of orthogonalization in
function spaces.

\begin{remark}[Bases in Hilbert function spaces]
\label{orthogonalityremark} A set of linearly independent functions $%
\mathcal{F}_{N}=\{f_{1},...,f_{N}\}\subset \mathcal{F}_{\mathbb{M}}^{N,%
\mathbb{M}}$ provide a fundamental tool in mathematics, since it guarantees
that each function contributes in a unique manner, most notably in function
representation. For example, consider the space of functions spanned by $%
\mathcal{F}_{N}$ defined by%
\begin{equation}
\overline{\mathcal{F}}_{N}=span\mathcal{F}_{N}=\left\{ f\in \mathcal{F}_{%
\mathbb{M}}^{N,\mathbb{M}}:f=\sum\limits_{i=1}^{N}c_{i}f_{i},c_{i}\in
\mathbb{M}\right\} .  \label{Span}
\end{equation}%
Obviously, finding a unique representation of $f\in \overline{\mathcal{F}}%
_{N},$ for specific constants $c_{i}$, requires an additional property, for
example, that $\mathcal{F}_{N}$ is orthogonal with respect to some inner
product $\rho $. Then the constants in expansion (\ref{Span}) can be
uniquely determined via inner products, and as a result, orthogonality is an
essential property in forming a basis in a Hilbert function space.\newline
Having a function basis allows for simpler representation and computation of
other functions. In particular, orthogonality simplifies calculations, such
as projections or expansions, and ensures that the functions do not overlap,
making them easier to work with in various mathematical applications. For
example, orthogonality is useful in solving differential equations and in
Fourier series, where orthogonal functions help in decomposing functions
into simpler components.
\end{remark}

Before we discuss Wronski orthogonalization and ODEs, we discuss how the
Gram-Schmidt process is a special case of the Wronski process.

\begin{remark}[Gram-Schmidt process via Wronski process]
\label{relationGS}Given a set of linearly independent functions $\mathcal{U}%
_{N}=\{u_{1},...,$ $u_{N}\}\subset \mathcal{F}_{\mathbb{M}}^{N,\mathbb{M}},$
the Gram-Schmidt process allows us to construct an orthogonal set $\mathcal{V%
}_{N}=\{v_{1},...,v_{N}\}\subset \mathcal{F}_{\mathbb{M}}^{N,\mathbb{M}},$
as follows; start by setting $v_{1}(x)=u_{1}(x),$ and build the other
functions in $\mathcal{V}_{N}$ using%
\begin{equation}
v_{k}(x)=u_{k}(x)-\sum\limits_{i=1}^{k-1}\frac{\rho (u_{k},v_{i})}{\rho
(v_{i},v_{i})}v_{i}(x),  \label{GS1}
\end{equation}%
for all $k=2,3,...,N$. Note the geometrical interpretation of the
construction, where in order to compute $v_{k},$ we project $u_{k}$
orthogonally onto the subspace $\overline{\mathcal{V}}_{k-1}$ generated by $%
\{v_{1},...,v_{k-1}\},$ which is the same as the subspace $\overline{%
\mathcal{U}}_{k-1}$ generated by $\{u_{1},...,u_{k-1}\}.$ The function $%
v_{k} $ is then defined to be the difference between $u_{k}$ and this
projection, which is guaranteed to be orthogonal to all functions in $%
\overline{\mathcal{V}}_{k-1}$.\newline
We show that regardless of how we have arrived at the set $\mathcal{V}_{N},$
there exist weight functions $h_{n},$ $n=1,2,...,N-1,$ such that starting
with $f_{1}(x)=u_{1}(x),$ we can construct $\mathcal{F}_{N}=%
\{f_{1},...,f_{N}\},$ using the Wronski process and obtain the same subspace
as the Gram-Schmidt process, i.e., $\mathcal{F}_{N}=\mathcal{V}_{N}$.\newline
To see how the Wronski process leads to $\mathcal{V}_{N}$, start with $k=2$,
and use (\ref{sol1}) to write%
\begin{equation*}
f_{2}(x)=F_{h_{1},f_{1}}(x)-\frac{f_{1}(x)}{\left\Vert f_{1}\right\Vert
_{\rho }^{2}}\rho \left( f_{1},F_{h_{1},f_{1}}\right) ,
\end{equation*}%
where in order for $f_{2}=v_{2},$ we must have that the weight function $%
h_{1}$ satisfies (\ref{Ffunc0}), with%
\begin{eqnarray*}
u_{2}(x) &=&F_{h_{1},f_{1}}(x)=f_{1}(x)\int\limits_{x_{0}}^{x}\frac{h_{1}(t)%
}{f_{1}(t)}d\mu (t) \\
&=&u_{1}(x)\int\limits_{x_{0}}^{x}\frac{h_{1}(t)}{u_{1}(t)}d\mu (t),
\end{eqnarray*}%
so that differentiating with respect to $x$ both sides above, we obtain%
\begin{equation*}
h_{1}(x)=u_{1}(x)\frac{d}{dx}\left( \frac{u_{2}(x)}{u_{1}(x)}\right) .
\end{equation*}%
For $k=3,$ in order for $f_{3}=v_{3},$ equation (\ref{sol1}) leads to%
\begin{eqnarray*}
f_{3}(x) &=&F_{h_{2},f_{1},f_{2}}(x)-\frac{f_{1}(x)}{\left\Vert
f_{1}\right\Vert _{\rho }^{2}}\rho \left( f_{1},F_{h_{1},f_{1},f_{2}}\right)
\\
&&-\frac{f_{2}(x)}{\left\Vert f_{2}\right\Vert _{\rho }^{2}}\rho \left(
f_{2},F_{h_{2},f_{1},f_{2}}\right) ,
\end{eqnarray*}%
so that $h_{2}$ is obtained via equation (\ref{Ffunc1}), with $f_{1}=v_{1},$
and $f_{2}=v_{2},$ as%
\begin{eqnarray*}
u_{3}(x) &=&F_{h_{2},f_{1},f_{2}}(x) \\
&=&f_{1}(x)\int\limits_{x_{0}}^{x}\frac{W_{1}(f_{1},f_{2})(t)}{%
W(f_{1},f_{2})(t)}h_{2}(t)d\mu (t) \\
&&+f_{2}(x)\int\limits_{x_{0}}^{x}\frac{W_{2}(f_{1},f_{2})(t)}{%
W(f_{1},f_{2})(t)}h_{2}(t)d\mu (t) \\
&=&-v_{1}(x)\int\limits_{x_{0}}^{x}\frac{v_{2}(t)}{W(v_{1},v_{2})(t)}%
h_{2}(t)d\mu (t) \\
&&+v_{2}(x)\int\limits_{x_{0}}^{x}\frac{v_{1}(t)}{W(v_{1},v_{2})(t)}%
h_{2}(t)d\mu (t),
\end{eqnarray*}%
such that%
\begin{gather*}
\int\limits_{x_{0}}^{x}\frac{v_{1}(t)}{W(v_{1},v_{2})(t)}h_{2}(t)d\mu (t) \\
-\frac{v_{1}(x)}{v_{2}(x)}\int\limits_{x_{0}}^{x}\frac{v_{2}(t)}{%
W(v_{1},v_{2})(t)}h_{2}(t)d\mu (t)=\frac{u_{3}(x)}{v_{2}(x)}.
\end{gather*}%
Differentiating with respect to $x$ both sides above yields%
\begin{gather*}
\frac{d}{dx}\left( \frac{u_{3}(x)}{v_{2}(x)}\right) =\frac{v_{1}(x)}{%
W(v_{1},v_{2})(x)}h_{2}(x) \\
-\frac{d}{dx}\left( \frac{v_{1}(x)}{v_{2}(x)}\right) \int\limits_{x_{0}}^{x}%
\frac{v_{2}(t)}{W(v_{1},v_{2})(t)}h_{2}(t)d\mu (t) \\
-\frac{v_{1}(x)}{v_{2}(x)}\frac{v_{2}(x)}{W(v_{1},v_{2})(x)}h_{2}(x),
\end{gather*}%
which leads to%
\begin{equation*}
\frac{d}{dx}\left( \frac{u_{3}(x)}{v_{2}(x)}\right) =-\frac{d}{dx}\left(
\frac{v_{1}(x)}{v_{2}(x)}\right) \int\limits_{x_{0}}^{x}\frac{v_{2}(t)}{%
W(v_{1},v_{2})(t)}h_{2}(t)d\mu (t).
\end{equation*}%
Isolating the integral and differentiating once again with respect to $x$
yields%
\begin{equation*}
\frac{v_{2}(x)}{W(v_{1},v_{2})(x)}h_{2}(x)=-\frac{d}{dx}\left( \frac{\frac{d%
}{dx}\left( \frac{u_{3}(x)}{v_{2}(x)}\right) }{\frac{d}{dx}\left( \frac{%
v_{1}(x)}{v_{2}(x)}\right) }\right) ,
\end{equation*}%
so that we obtain%
\begin{equation*}
h_{2}(x)=-\frac{d}{dx}\left( \frac{\frac{d}{dx}\left( \frac{u_{3}(x)}{%
v_{2}(x)}\right) }{\frac{d}{dx}\left( \frac{v_{1}(x)}{v_{2}(x)}\right) }%
\right) \frac{W(v_{1},v_{2})(x)}{v_{2}(x)}.
\end{equation*}%
Continue in a similar fashion to obtain all the functions $%
\{h_{1},...,h_{N-1}\}$ which lead to a constructed set $\mathcal{F}%
_{N}=\{f_{1},...,f_{N}\},$ that satisfies $\mathcal{F}_{N}=\mathcal{V}_{N}$.
As a result, the Wronski process (which is agnostic about how one obtained
the linearly independent set $\mathcal{U}_{N}$ we start with), is flexible
enough to adjust the weight functions and produce the set $\mathcal{V}_{N}$\
constructed via the Gram-Schmidt process, and therefore it is a
generalization of the famed process.
\end{remark}

\section{Solutions to Ordinary Differential Equations}

\label{ExplicitSOlODE}In order to illustrate the importance of the proposed
Wronski orthogonalization construction, we present a first application of
the method in mathematics involving solutions to ODEs.

Let $c_{j}^{k}=\frac{k!}{(k-j)!j!},$ denote the binomial coefficient, and
take $L_{m}[\cdot ]$ to be the $m$th order linear, differential operator
defined by%
\begin{equation}
L_{m}[y]=p_{m}(x)\frac{d^{m}y}{dx^{m}}+p_{m-1}(x)\frac{d^{m-1}y}{dx^{m-1}}%
+...+p_{1}(x)\frac{dy}{dx}+p_{0}(x)y,  \label{LinOperatornthorder}
\end{equation}%
where the $\{p_{i}(x)\}$ are $\mathbb{M}$-valued functions and $p_{m}(x)\neq
0,$ on an interval $x\in I=(a,b)\subset \mathbb{M}$, and recall the
homogeneous ODE (HODE) given by%
\begin{equation}
L_{m}[y]=\sum\limits_{i=0}^{m}p_{i}(x)\frac{d^{i}y(x)}{dx^{i}}=0.
\label{homo1}
\end{equation}%
Define $m$ initial conditions%
\begin{equation}
C_{j}[y]:=\frac{d^{j}y}{dx^{j}}(x_{0})=A_{j},\text{ }j=0,1,2,...,m-1
\label{Initial1}
\end{equation}%
for some constants $A_{j}\in \mathbb{M},$ and $x_{0}\in I$.

We consider the linear, inhomogeneous ODE problem, of the form%
\begin{equation}
L_{m}[y]=h(x),\text{ under conditions }\mathcal{C},  \label{LODE1}
\end{equation}%
where $\mathcal{C}$ is shorthand for all conditions in (\ref{Initial1})$.$
Note here that if $\{p_{i}(x)\}_{i=0}^{m}$ and $h(x)$ satisfy certain
conditions, then there exists a unique solution to the problem (\ref{LODE1}%
); more precisely, when $\mathbb{M}=\Re ,$ continuity suffices (e.g., see
\cite{kelley2010theory}, Theorem 6.2), while for $\mathbb{M}=\mathbb{C}$,
one requires the functions to be analytic (e.g., see \cite%
{teschl2012ordinary}, Theorem 4.1) or have additional requirements on $L_{m}$
(e.g., see self-adjoint eigenfunction problems, \cite{zwillinger2021handbook}%
, pg 95).

The following theorem presents the solutions of the problem (\ref{LODE1}) in
terms of a fundamental and orthogonal set of solutions constructed by
Theorem \ref{thmwronski}, and based on a non-trivial solution of (\ref{homo1}%
). The proof is given in the Appendix (Section \ref{App1}). We will write $W(%
\mathcal{F}_{m})$ as a short hand for $W(y_{1},...,y_{m}),$ and similarly, $%
W_{k}(\mathcal{F}_{m})$ for $W_{k}(y_{1},...,y_{m}),$ where $\mathcal{F}%
_{m}=\{y_{1},...,y_{m}\}.$

\begin{theorem}
Assume that $\rho $ is an inner product such that $(\mathcal{F}_{\mathbb{M}%
}^{m,\mathbb{M}},\rho )$ is a Hilbert function space, and consider the
linear, inhomogeneous ODE problem (\ref{LODE1}). Further, assume that there
exists a non-zero function $y_{1}\in \mathcal{F}_{\mathbb{M}}^{m,\mathbb{M}%
}, $ with $0<\left\Vert y_{1}\right\Vert _{\rho }<+\infty ,$ such that $%
y_{1} $ is a non-trivial solution to (\ref{homo1}). Then\label%
{thmwronskiODEs} the set $\mathcal{F}_{m}=\{y_{1},$ $...,y_{m}\}$ $\subset
\mathcal{F}_{\mathbb{M}}^{m,\mathbb{M}}$, of linearly independent and
orthogonal functions with respect to $\rho ,$ constructed in Theorem \ref%
{thmwronski} based on this $y_{1},$ forms the set of fundamental and
orthogonal solutions to (\ref{homo1}), provided that conditions
\begin{gather}
\sum\limits_{r=0}^{m-1-u}c_{r}^{u+r}\sum%
\limits_{i=u+r+1}^{m}p_{i}(x)c_{u+r+1}^{i}\sum\limits_{j=1}^{k-1}\frac{%
\partial ^{i-u-r-1}y_{j}(x)}{\partial ^{i-u-r-1}x}  \label{CondhHODE} \\
\frac{d^{r}}{dx^{r}}\left[ \frac{W_{j}(\mathcal{F}_{k-1})(x)}{W(\mathcal{F}%
_{k-1})(x)}\right] \in \mathcal{F}_{\mathbb{M}}^{m,\mathbb{M}},  \notag
\end{gather}%
for all $u=0,1,...,m-1$, hold and $h_{k-1}$ is chosen to be a solution to
the HODE\
\begin{gather}
\sum\limits_{u=0}^{m-1}\sum\limits_{r=0}^{m-1-u}c_{r}^{u+r}\sum%
\limits_{i=u+r+1}^{m}p_{i}(x)c_{u+r+1}^{i}\sum\limits_{j=1}^{k-1}\frac{%
\partial ^{i-u-r-1}y_{j}(x)}{\partial ^{i-u-r-1}x}  \label{HkHODE} \\
\frac{d^{r}}{dx^{r}}\left[ \frac{W_{j}(\mathcal{F}_{k-1})(x)}{W(\mathcal{F}%
_{k-1})(x)}\right] \frac{d^{u}h_{k-1}(x)}{d^{u}x}=0,  \notag
\end{gather}%
for each $k=2,...,m.$ The totality of solutions to problem (\ref{LODE1}) is
then obtained using the process described in Remark \ref%
{variationofparameters}, i.e.,%
\begin{equation*}
y(x)=y_{0}(x)+\sum\limits_{i=1}^{m}c_{i}y_{i}(x),
\end{equation*}%
$x\in I\subset \mathbb{M}$, where $y_{0}(x)$ is a specific solution of (\ref%
{LODE1}), and $c_{1},c_{2},...,c_{n}$ are real constants that are determined
using the initial conditions $\mathcal{C}$. The specific solution can be
written in terms of Wronskians as%
\begin{equation*}
y_{0}(x)=\sum\limits_{k=1}^{m}y_{k}(x)\int\limits_{x_{0}}^{x}\frac{W_{k}(%
\mathcal{F}_{m})(t)}{W(\mathcal{F}_{m})(t)}\frac{h(t)}{p_{m}(t)}d\mu (t),
\end{equation*}%
for any $x_{0}\in I$.
\end{theorem}

We write $y^{\prime },$ $y^{\prime \prime },$ and so forth, for the
different order derivatives of $y$ wrt $x$. We collect the solutions to the
second order HODE explicitly, in the following remark.

\begin{remark}[Second order HODEs]
\label{2ndorderHODE} Let $m=2,$ and consider the homogeneous ODE%
\begin{equation}
p_{2}y^{\prime \prime }+p_{1}y^{\prime }+p_{0}y=0,  \label{Homo2rndorder}
\end{equation}%
with $p_{2}\neq 0.$ In this case, we need to find $h_{1}$ that is a solution
to the HODE (\ref{HkHODE}), which for $m=2$ reduces to%
\begin{equation*}
p_{2}h_{1}^{\prime }+\left[ p_{1}+p_{2}\frac{y_{1}^{\prime }}{y_{1}}\right]
h_{1}=0.
\end{equation*}%
Multiplying both sides with $y_{1},$ we obtain the HODE%
\begin{equation}
p_{2}y_{1}h_{1}^{\prime }+\left[ p_{1}y_{1}+p_{2}y_{1}^{\prime }\right]
h_{1}=0  \label{2ndorderHODE1}
\end{equation}%
with solution%
\begin{equation}
h_{1}(x)=\frac{1}{\left\vert y_{1}(x)\right\vert }\exp \left\{
-\int\limits_{x_{0}}^{x}\frac{p_{1}(s)}{p_{2}(s)}d\mu (s)\right\} \neq 0.
\label{2ndorderHODE1h1}
\end{equation}%
Consequently, using Equation (\ref{SolLinInded1}), we find
\begin{equation}
y_{2}(x)=\left[ \int\limits_{x_{0}}^{x}\frac{1}{\left( y_{1}(t)\right) ^{2}}%
e^{-\int\limits_{x_{0}}^{t}\frac{p_{1}(s)}{p_{2}(s)}d\mu (s)}d\mu
(t)-\lambda _{1}\right] y_{1}(x),  \label{LinIndepy_2}
\end{equation}%
where $\lambda _{1}$ is given by \ref{Lambdajs}, and $\mathcal{F}%
_{2}=\{y_{1},y_{2}\}$ is the set of fundamental and orthogonal set of
solutions to the HODE\ (\ref{Homo2rndorder}).\newline
In order to appreciate the underlying details that might be lost in the
presentation of the general approach of Equation (\ref{HkHODE}), we will
reach the same answer via the reduction of order method. Let $y_{1}$ be a
known solution to (\ref{Homo2rndorder}), with $y_{2}$ given by (\ref%
{SolLinInded1}). We will show that there exists $h_{1}\neq 0,$ such that $%
\mathcal{F}_{2}=\{y_{1},y_{2}\}$ as defined via $y_{1}$ and $h_{1},$ form
the set of fundamental and orthogonal solutions of the HODE (\ref%
{Homo2rndorder}). Then we have%
\begin{equation*}
y_{2}(x)=\left[ G_{1,1}(x)-\lambda _{1}\right] y_{1}(x)=\left[
\int\limits_{x_{0}}^{x}\frac{h_{1}(t)}{y_{1}(t)}d\mu (t)-\lambda _{1}\right]
y_{1}(x),
\end{equation*}%
$x,x_{0}\in I,$ such that (\ref{Homo2rndorder}) is satisfied. Taking
derivatives of $y_{2}$ with respect to $x$, yields%
\begin{equation}
y_{2}^{\prime }=G_{1,1}^{\prime }y_{1}+\left[ G_{1,1}-\lambda _{1}\right]
y_{1}^{\prime },  \label{firstderiv1}
\end{equation}%
and%
\begin{equation}
y_{2}^{\prime \prime }=G_{1,1}^{\prime \prime }y_{1}+2G_{1,1}^{\prime
}y_{1}^{\prime }+\left[ G_{1,1}-\lambda _{1}\right] y_{1}^{\prime \prime },
\label{secondderiv1}
\end{equation}%
so that%
\begin{eqnarray*}
0 &=&p_{2}G_{1,1}^{\prime \prime }y_{1}+2p_{2}G_{1,1}^{\prime }y_{1}^{\prime
}+p_{2}\left[ G_{1,1}-\lambda _{1}\right] y_{1}^{\prime \prime } \\
&&+p_{1}G_{1,1}^{\prime }y_{1}+p_{1}\left[ G_{1,1}-\lambda _{1}\right]
y_{1}^{\prime }+p_{0}\left[ G_{1,1}-\lambda _{1}\right] y_{1} \\
&=&\left[ G_{1,1}-\lambda _{1}\right] \left[ p_{2}y_{1}^{\prime \prime
}+p_{1}y_{1}^{\prime }+p_{0}y_{1}\right] + \\
&&p_{2}G_{1,1}^{\prime \prime }y_{1}+2p_{2}G_{1,1}^{\prime }y_{1}^{\prime
}+p_{1}G_{1,1}^{\prime }y_{1},
\end{eqnarray*}%
and since $y_{1}$ is a solution of (\ref{Homo2rndorder}), we must have%
\begin{equation}
p_{2}y_{1}G_{1,1}^{\prime \prime }+\left[ 2p_{2}y_{1}^{\prime }+p_{1}y_{1}%
\right] G_{1,1}^{\prime }=0.  \label{1rstorderredorder}
\end{equation}%
Now note that%
\begin{equation*}
G_{1,1}^{\prime }=\frac{h_{1}}{y_{1}},\text{ and }G_{1,1}^{\prime \prime }=%
\frac{h_{1}^{\prime }y_{1}-h_{1}y_{1}^{\prime }}{y_{1}^{2}},
\end{equation*}%
so that using the latter in (\ref{1rstorderredorder}), leads to%
\begin{equation*}
p_{2}y_{1}h_{1}^{\prime }+\left[ p_{2}y_{1}^{\prime }+p_{1}y_{1}\right]
h_{1}=0,
\end{equation*}%
with the solution given in Equation (\ref{2ndorderHODE1h1}). Consequently,
the proposed set $\mathcal{F}_{2}=\{y_{1},y_{2}\}$ forms a set of
fundamental and orthogonal solutions to (\ref{Homo2rndorder}). In addition,
we note that $\mathcal{F}_{2}$ is a basis for $\overline{\mathcal{F}}_{2},$
with the basis elements obtained by solving the HODE (\ref{2ndorderHODE1}).
Finally, using $\mathcal{F}_{2}$ we can apply Remark \ref%
{variationofparameters} for $m=2$ and obtain the totality of solutions to
the problem (\ref{Inhomonorder}).
\end{remark}

In order to illustrate the importance of the latter theorem, we consider
some special cases in the following example.

\begin{example}[Some classic ODEs]
\label{ODEexample}We consider problem (\ref{LODE1}) and the construction of
a general solution via Wronski orthogonalization, under different scenarios.
In what follows, we use the standard inner product%
\begin{equation}
\rho _{2}(f_{1},f_{2})=\int\limits_{a}^{b}f_{1}(x)f_{2}(x)d\mu (x),
\label{Rho_2innerproduct}
\end{equation}%
over some interval $I=[a,b]\subset \Re ,$ with the additional assumptions $%
\rho _{2}(f_{i},f_{i})^{2}=\left\Vert f_{i}\right\Vert
_{2}^{2}=\int\limits_{\Re }|f_{i}|^{2}d\mu (x)<+\infty ,$ $i=1,2$, i.e., the
Lebesgue $\mathcal{L}^{2}([a,b])$-space.

\begin{enumerate}
\item Airy's equation: Assume that $m=2$, and consider Airy's equation given
by%
\begin{equation}
y^{\prime \prime }-xy=0,\text{ }x>0.  \label{Airyeq1}
\end{equation}%
It is known (see \cite{zwillinger2021handbook}, p. 339), that the Airy
functions%
\begin{equation*}
Ai(x)=\frac{1}{\pi }\int\limits_{0}^{\infty }\cos \left( \frac{t^{3}}{3}%
+xt\right) d\mu (t),
\end{equation*}%
of the first kind, and%
\begin{equation*}
Bi(x)=\frac{1}{\pi }\int\limits_{0}^{\infty }e^{-\frac{t^{3}}{3}+xt}\cos
\left( \frac{t^{3}}{3}+xt\right) d\mu (t),
\end{equation*}%
of the second kind, are two linearly independent solutions of (\ref{Airyeq1}%
).\newline
We present two new fundamental sets of solutions via Wronski
orthogonalization. Starting with $Ai(x)$ and applying Remark \ref%
{2ndorderHODE}, with $x_{0}>0$, $p_{2}(x)=1,$ $p_{1}(x)=0,$ $p_{0}(x)=-x,$
and $y_{1}(x)=$ $Ai(x),$ $x>0.$ Then from Equation (\ref{LinIndepy_2}) we
have%
\begin{gather*}
y_{2}(x)=\left[ \int\limits_{x_{0}}^{x}\frac{1}{\left( Ai(t)\right) ^{2}}%
d\mu (t)-\lambda _{1}\right] Ai(x) \\
=Ai(x)\int\limits_{x_{0}}^{x}\frac{1}{\left( Ai(t)\right) ^{2}}d\mu (t) \\
-\frac{Ai(x)}{\left\Vert Ai\right\Vert _{\rho _{2}}^{2}}\int\limits_{0}^{+%
\infty }\left( Ai(x)\right) ^{2}\int\limits_{x_{0}}^{x}\frac{1}{\left(
Ai(t)\right) ^{2}}d\mu (t)d\mu (x),
\end{gather*}%
since $\lambda _{1}=\rho _{2}\left( Ai,F_{h_{1},Ai}\right) /\left\Vert
Ai\right\Vert _{\rho _{2}}^{2},$ where%
\begin{equation*}
F_{h_{1},Ai}(x)=Ai(x)\int\limits_{x_{0}}^{x}\frac{1}{\left( Ai(t)\right) ^{2}%
}d\mu (t),
\end{equation*}%
and%
\begin{equation*}
\rho _{2}\left( Ai,F_{h_{1},Ai}\right) =\int\limits_{0}^{+\infty }\left(
Ai(x)\right) ^{2}\int\limits_{x_{0}}^{x}\frac{1}{\left( Ai(t)\right) ^{2}}%
d\mu (t)d\mu (x),
\end{equation*}%
so that $\{y_{1},y_{2}\}$ forms a fundamental and orthogonal set of
solutions of the ODE (\ref{Airyeq1}).\newline
On the other hand, starting with the solution $y_{1}(x)=Bi(x),$ $x>0,$ the
second solution of the fundamental and orthogonal set is given by%
\begin{equation*}
y_{2}(x)=\left[ \int\limits_{x_{0}}^{x}\frac{1}{\left( Bi(t)\right) ^{2}}%
d\mu (t)-\lambda _{1}\right] Bi(x),
\end{equation*}%
with $\lambda _{1}=\rho _{2}\left( Bi,F_{h_{1},Bi}\right) /\left\Vert
Bi\right\Vert _{\rho _{2}}^{2},$ where%
\begin{equation*}
F_{h_{1},Bi}(x)=Bi(x)\int\limits_{x_{0}}^{x}\frac{1}{\left( Bi(t)\right) ^{2}%
}d\mu (t),
\end{equation*}%
and%
\begin{equation*}
\rho _{2}\left( Bi,F_{h_{1},Bi}\right) =\int\limits_{0}^{+\infty }\left(
Bi(x)\right) ^{2}\int\limits_{x_{0}}^{x}\frac{1}{\left( Bi(t)\right) ^{2}}%
d\mu (t)d\mu (x).
\end{equation*}%
Airy functions are applied in many branches of classical and quantum physics
(see \cite{vallee2010airy}). They are associated with wave equations with
turning points, for which asymptotic solutions are exponential on one side
and oscillatory on the other. Note what is accomplished by the Wronski
process in this case; it offers two alternative orthogonal and linearly
independent functions to the classic Airy functions $Ai(x)$ and $Bi(x),$
that can be used by physicists to study wave motion, particularly in
diffraction theory.

\item Euler Equations: Assume that $m=2$, and consider the HODE%
\begin{equation}
x^{2}y^{\prime \prime }-xy^{\prime }+y=0,  \label{3rdorderex}
\end{equation}%
$x\in I=(0,1]$. It is easy to verify that the set $\{y_{1},y_{2}\},$ with $%
y_{1}(x)=x,$ $y_{2}(x)=x\log x,$ forms a set of fundamental solutions to (%
\ref{3rdorderex}), because $W(y_{1},y_{2})=x\neq 0,$ $x\in I$. Then the
general solution is given by $y(x)=c_{1}x+c_{2}x\log x,$ where $%
c_{1},c_{2}\in \Re .$\newline
However, using the standard inner product (\ref{Rho_2innerproduct}), the
functions $\{x,$ $x\log x\}$ are not orthogonal since $\rho _{2}(x,$ $x\log
x)=2/9$.\newline
Using Theorem \ref{thmwronski}, we build a set of fundamental solutions that
are also orthogonal. Starting with $y_{1}(x)=x,$ we have $\left\Vert
y_{1}\right\Vert _{\rho _{2}}^{2}=\frac{1}{3}>0,$ which is finite. We will
find the function $h_{1},$ in order for $\mathcal{F}_{2}=\{y_{1},y_{2}\}$ to
emerge as the set of fundamental and orthogonal solutions. From Remark \ref%
{2ndorderHODE}, we have the HODE (\ref{2ndorderHODE1}) which reduces to%
\begin{equation}
x^{3}h_{1}^{\prime }=0\Rightarrow h_{1}^{\prime }=0,
\end{equation}%
with solution $h_{1}(x)=c,$ for some real constant $c\neq 0.$ Then, it is
straightforward to see that for some $x_{0}\in I,$ we have%
\begin{equation*}
G_{1,1}(x)=c\log \frac{x}{x_{0}},
\end{equation*}%
and%
\begin{eqnarray*}
\lambda _{1} &=&\frac{1}{\left\Vert y_{1}\right\Vert _{\rho _{2}}^{2}}\rho
_{2}\left( y_{1},F_{h_{1},y_{1}}\right) =3\rho _{2}\left(
y_{1},y_{1}G_{1,1}\right) \\
&=&-c(1+\log x_{0}),
\end{eqnarray*}%
which leads to%
\begin{equation*}
y_{2}(x)=\left[ G_{1,1}(x)-\lambda _{1}\right] y_{1}(x)=c\left[ \log x+1%
\right] x.
\end{equation*}%
Now we can verify orthogonality using (\ref{Rho_2innerproduct}); we have
\begin{eqnarray*}
\rho _{2}(y_{1},y_{2}) &=&\int\limits_{0}^{1}x^{2}c\left[ \log x+1\right]
xd\mu (x) \\
&=&-\frac{c}{3}-\frac{c}{3}\log x_{0}+\frac{c}{3}+\frac{c}{3}\log x_{0}=0.
\end{eqnarray*}%
As a result, the fundamental and orthogonal set of solutions is $\{x,c\left[
\log x+1\right] x\},$ for some real constant $c\neq 0$. Note that the
Wronski orthogonalization process does not lead to a unique fundamental and
orthogonal set of solutions.
\end{enumerate}
\end{example}

Based on the theoretical development in this section, we can see that the
Wronski orthogonalization process when applied to ODEs, provides a set of
linearly independent and orthogonal set of functions $\mathcal{F}_{N}$ as
the fundamental set of solutions of a HODE. Combining these two important
properties makes $\mathcal{F}_{N}$ a basis for the spanned space $\overline{%
\mathcal{F}}_{N}$ of $\mathcal{F}_{N}.$ Therefore, this observation leads to
the following general conjecture regarding bases functions in Hilbert
function spaces.

\begin{conjecture}[Bases as solutions to HODEs]
\label{ConjBasis}Let $(\mathcal{F}_{\mathbb{M}}^{\mathbb{M}},\rho )$ be a
Hilbert function space. Assume that $\mathcal{E}=\{e_{i}\}_{i\in I}\subset
\mathcal{F}_{\mathbb{M}}^{\mathbb{M}}$, with $\overline{\mathcal{E}}=%
\mathcal{F}_{\mathbb{M}}^{\mathbb{M}},$ is an orthonormal basis of $\mathcal{%
F}_{\mathbb{M}}^{\mathbb{M}},$ where the index set $I$ is countable or
uncountable. Then, there exists a unique linear HODE (of potentially
infinite order) such that its fundamental set of solutions via Wronski
orthogonalization is the set $\mathcal{E}.$
\end{conjecture}

Since the basis function set $\mathcal{E}$ is constructed via Wronski
orthogonalization, it will be aptly called the Wronski basis of $\mathcal{F}%
_{\mathbb{M}}^{\mathbb{M}}.$ In view of Theorem \ref{thmwronskiODEs}, Remark %
\ref{2ndorderHODE} and Example \ref{ODEexample}, the conjecture is satisfied
trivially when $I$ is a finite set. More precisely, the variation of
parameters method (Remark \ref{variationofparameters}) tells us exactly the
form of the unique HODE; when $I$ is a finite set, say without loss of
generality $I=\{1,2,...,N\},$ then the HODE is given by Equation (\ref%
{HomogeneousWronski}), i.e.,%
\begin{equation}
\frac{W(\mathcal{E}_{N},y)}{W(\mathcal{E}_{N})}=\frac{%
W(e_{1},e_{2},...,e_{N},y)}{W(e_{1},e_{2},...,e_{N})}=0,  \label{basisHode}
\end{equation}%
a HODE in $y$, with its set of fundamental and orthogonal solutions being
the basis $\mathcal{E}_{N}=\{e_{1},...,e_{N}\}$ of the spanned space $%
\overline{\mathcal{E}}_{N}$. Let us rewrite the Wronskian in terms of the
determinant expansion along the $(N+1)$-column, i.e., by expanding with
respect to the column corresponding to $y$ via%
\begin{gather}
W(\mathcal{E}_{N},y)=\left\vert
\begin{tabular}{llll}
$e_{1}(x)$ & $...$ & $e_{N}(x)$ & $y(x)$ \\
$\frac{de_{1}(x)}{dx}$ & $...$ & $\frac{de_{N}(x)}{dx}$ & $\frac{dy(x)}{dx}$
\\
$...$ &  & $...$ &  \\
$\frac{d^{N-1}e_{1}(x)}{dx^{N-1}}$ & $...$ & $\frac{d^{N-1}e_{N}(x)}{dx^{N-1}%
}$ & $\frac{d^{N-1}y(x)}{dx^{N-1}}$ \\
$\frac{d^{N}e_{1}(x)}{dx^{N}}$ & $...$ & $\frac{d^{N}e_{N}(x)}{dx^{N}}$ & $%
\frac{d^{N}y(x)}{dx^{N}}$%
\end{tabular}%
\right\vert  \notag \\
=\sum\limits_{i=1}^{N+1}(-1)^{i+N+1}M_{i,N+1}(\mathcal{E}_{N},y)\frac{%
d^{i-1}y(x)}{dx^{i-1}},  \label{WronskiHODE1}
\end{gather}%
where $M_{ij}(\mathcal{E}_{N},y)$ the $(i,j)$-minor$.$

In order to generalize to the countably infinite case, we send $N\rightarrow
\infty ,$ so that%
\begin{equation}
W(\mathcal{E}_{\infty },y)=\underset{N\rightarrow \infty }{\lim }%
\sum\limits_{i=1}^{N+1}(-1)^{i+N+1}M_{i,N+1}(\mathcal{E}_{N},y)\frac{%
d^{i-1}y(x)}{dx^{i-1}},  \label{InfiniteWronskian}
\end{equation}%
which will be well defined if the following conditions hold;

\textbf{Wronski basis conditions:}

\begin{enumerate}
\item the functions $y,e_{1},e_{2},...,$ are analytic,

\item $W(\mathcal{E}_{N},y)$ exists, is non-zero, and is finite for all $%
N\in \mathbb{N}$, and

\item the infinite sum in Equation (\ref{InfiniteWronskian}) exists and is
finite.
\end{enumerate}

The countable case and uncountable case for the index set $I$ requires
further investigation and will be tackled elsewhere.

\section{Concluding Remarks}

The proposed Wronski orthogonalization process was illustrated to be a more
streamlined alternative procedure to the classic Gram-Schmidt
orthogonalization process, since we do not require a set of linearly
independent elements of the Hilbert space before we begin the construction.
In addition, the proposed method provides great flexibility, since we can
choose any weight functions we wish, under mild assumptions, and tailor them
to the needs of the context in which the process is used.

There are many applications that were not presented in this paper, since our
purpose was to introduce the new method. In particular, the proposed
construction paves the way for the creation, study and application of new
bases in Hilbert function spaces. Moreover, many classic orthogonal
functions for specific Hilbert function spaces, are special cases for
appropriate choice of the weight functions, e.g., classic polynomials. These
are subjects of great interest in further illustrating the importance of the
Wronski orthogonalization process, and will be investigated elsewhere.

\section*{Acknowledgments}

I am grateful to Professor Petros Valettas, Department of Mathematics,
University of Missouri, for his constructive comments and suggestions on an
earlier version of the manuscript.

\section*{Statements and Declarations}

I have no conflict of interest with the results in this paper. This research
received no external funding. There is no data available or required by this
paper.

\section{Appendix}

\subsection{Proof of Theorem \protect\ref{thmwronskiODEs}}

\label{App1}Assume that there exists $y_{1}\in \mathcal{F}_{\mathbb{M}}^{m,%
\mathbb{M}},$ with $0<\left\Vert f_{1}\right\Vert _{\rho }<+\infty .$ Using
Theorem \ref{thmwronski}, we can build a set of linearly independent
functions $\mathcal{F}_{m}$ $=\{y_{1},...,y_{m}\}$ $\subset \mathcal{F}_{%
\mathbb{M}}^{m,\mathbb{M}}$, where%
\begin{eqnarray}
y_{k}(x)
&=&F_{h_{k-1},y_{1},y_{2},...,y_{k-1}}(x)-\sum\limits_{j=1}^{k-1}\lambda
_{j}y_{j}(x)  \label{SolLinInded1} \\
&=&\sum\limits_{j=1}^{k-1}\left[ G_{j,k-1}(x)-\lambda _{j}\right]
y_{j}(x)\neq 0,  \notag
\end{eqnarray}%
with%
\begin{equation}
F_{h_{k-1},y_{1},y_{2},...,y_{k-1}}(x)=\sum%
\limits_{j=1}^{k-1}G_{j,k-1}(x)y_{j}(x),  \label{Ffuncs}
\end{equation}%
\begin{equation}
G_{j,k-1}(x)=\int\limits_{x_{0}}^{x}\frac{W_{j}(\mathcal{F}_{k-1})(t)}{W(%
\mathcal{F}_{k-1})(t)}h_{k-1}(t)d\mu (t),  \label{Gfuncs}
\end{equation}%
and%
\begin{equation}
\lambda _{j}=\rho \left( y_{j},F_{h_{k-1},y_{1},y_{2},...,y_{k-1}}\right)
/\left\Vert y_{j}\right\Vert _{\rho }^{2},  \label{Lambdajs}
\end{equation}%
$j=1,...,k-1,$ $k=2,...,m.$ The set $\{y_{2},...,y_{m}\}$ are not
necessarily solutions to (\ref{homo1}), however, we will find conditions on
the functions $\mathcal{H}=\{h_{1},...,h_{m-1}\},$ in order for $\mathcal{F}%
_{m}$ to emerge as the set of fundamental and orthogonal solutions of the
HODE (\ref{homo1}).

Assume that $\mathcal{F}_{m}$ denotes a set of fundamental solutions to (\ref%
{homo1}), i.e., $\mathcal{F}_{m}$ are linearly independent and $%
L_{m}[y_{1}]=...=L_{m}[y_{m}]=0.$ Note that what follows is essentially the
reduction of order method (see \cite{zwillinger2021handbook}). By linearity
of $L_{m},$ using Equation (\ref{SolLinInded1}) in Equation (\ref%
{LinOperatornthorder}), yields%
\begin{eqnarray*}
L_{m}[y_{k}] &=&\sum\limits_{i=0}^{m}p_{i}(x)\frac{d^{i}y_{k}(x)}{dx^{i}} \\
&=&\sum\limits_{i=0}^{m}p_{i}(x)\frac{%
d^{i}F_{h_{k-1},y_{1},y_{2},...,y_{k-1}}(x)}{dx^{i}} \\
&&-\sum\limits_{i=0}^{m}p_{i}(x)\sum\limits_{j=1}^{k-1}\lambda _{j}\frac{%
d^{i}y_{j}(x)}{dx^{i}} \\
&=&\sum\limits_{i=0}^{m}p_{i}(x)\frac{%
d^{i}F_{h_{k-1},y_{1},y_{2},...,y_{k-1}}(x)}{dx^{i}} \\
&&-\sum\limits_{j=1}^{k-1}\lambda _{j}\sum\limits_{i=0}^{m}p_{i}(x)\frac{%
d^{i}y_{j}(x)}{dx^{i}} \\
&=&\sum\limits_{i=0}^{m}p_{i}(x)\frac{%
d^{i}F_{h_{k-1},y_{1},y_{2},...,y_{k-1}}(x)}{dx^{i}} \\
&&-\sum\limits_{j=1}^{k-1}\lambda _{j}L_{m}[y_{j}],
\end{eqnarray*}%
and since $L_{m}[y_{j}]=0,$ $j=1,...,k-1$, we need to find a set of
functions $\mathcal{H}$ such that%
\begin{equation}
L_{m}[y_{k}]=\sum\limits_{i=0}^{m}p_{i}(x)\frac{%
d^{i}F_{h_{k-1},y_{1},y_{2},...,y_{k-1}}(x)}{dx^{i}}=0,
\label{FunctionSystem1}
\end{equation}%
for all $k=2,...,m.$ Using Equation (\ref{Ffuncs}) we can rewrite (\ref%
{FunctionSystem1}) as%
\begin{eqnarray*}
0 &=&\sum\limits_{i=0}^{m}p_{i}(x)\frac{d^{i}}{dx^{i}}\left[
\sum\limits_{j=1}^{k-1}G_{j,k-1}(x)y_{j}(x)\right] \\
&=&\sum\limits_{i=0}^{m}p_{i}(x)\sum\limits_{j=1}^{k-1}\frac{d^{i}\left[
G_{j,k-1}(x)y_{j}(x)\right] }{dx^{i}},
\end{eqnarray*}%
for all $k=2,...,m.$

An appeal to the general Leibniz rule for differentiation gives%
\begin{equation}
\frac{d^{i}\left[ G_{j,k-1}(x)y_{j}(x)\right] }{dx^{i}}=\sum%
\limits_{l=0}^{i}c_{l}^{i}\frac{d^{l}G_{j,k-1}(x)}{dx^{l}}\frac{%
d^{i-l}y_{j}(x)}{d^{i-l}x},  \label{LeibnizRule1}
\end{equation}%
so that separating the term for $l=0$, we obtain%
\begin{eqnarray*}
0
&=&\sum\limits_{i=0}^{m}p_{i}(x)\sum\limits_{j=1}^{k-1}\sum%
\limits_{l=0}^{i}c_{l}^{i}\frac{d^{i-l}y_{j}(x)}{d^{i-l}x}\frac{%
d^{l}G_{j,k-1}(x)}{dx^{l}} \\
&=&\sum\limits_{j=1}^{k-1}\sum\limits_{i=0}^{m}p_{i}(x)\sum%
\limits_{l=0}^{i}c_{l}^{i}\frac{d^{i-l}y_{j}(x)}{d^{i-l}x}\frac{%
d^{l}G_{j,k-1}(x)}{dx^{l}} \\
&=&\sum\limits_{j=1}^{k-1}G_{j,k-1}(x)\sum\limits_{i=0}^{m}p_{i}(x)c_{0}^{i}%
\frac{d^{i}y_{j}(x)}{d^{i}x} \\
&&+\sum\limits_{j=1}^{k-1}\sum\limits_{l=1}^{m}\frac{d^{l}G_{j,k-1}(x)}{%
dx^{l}}\sum\limits_{i=l}^{m}p_{i}(x)c_{l}^{i}\frac{d^{i-l}y_{j}(x)}{d^{i-l}x}%
,
\end{eqnarray*}%
where by assumption%
\begin{equation*}
\sum\limits_{i=0}^{m}p_{i}(x)c_{0}^{i}\frac{d^{i}y_{j}(x)}{d^{i}x}%
=\sum\limits_{i=0}^{m}p_{i}(x)\frac{d^{i}y_{j}(x)}{d^{i}x}=L_{m}[y_{j}]=0,
\end{equation*}%
for all $j=1,...,k-1$. Consequently, Equation (\ref{FunctionSystem1}) is
written as%
\begin{equation*}
\sum\limits_{j=1}^{k-1}\sum\limits_{l=1}^{m}\frac{d^{l}G_{j,k-1}(x)}{dx^{l}}%
\sum\limits_{i=l}^{m}p_{i}(x)c_{l}^{i}\frac{d^{i-l}y_{j}(x)}{d^{i-l}x}=0.
\end{equation*}%
Now define%
\begin{equation*}
W_{j,k-1}(x)=\frac{dG_{j,k-1}(x)}{dx}=\frac{W_{j}(\mathcal{F}_{k-1})(x)}{W(%
\mathcal{F}_{k-1})(x)}h_{k-1}(x),
\end{equation*}%
so that we reduce the order by one, and consider the homogeneous ODE in $%
h_{k-1}$ given by%
\begin{equation*}
\sum\limits_{j=1}^{k-1}\sum\limits_{l=0}^{m-1}\sum%
\limits_{i=l+1}^{m}p_{i}(x)c_{l+1}^{i}\frac{\partial ^{i-l-1}y_{j}(x)}{%
\partial ^{i-l-1}x}\frac{\partial ^{l}W_{j,k-1}(x)}{\partial x^{l}}=0,
\end{equation*}%
and after some rearranging%
\begin{equation*}
\sum\limits_{l=0}^{m-1}\sum\limits_{i=l+1}^{m}p_{i}(x)c_{l+1}^{i}\sum%
\limits_{j=1}^{k-1}\frac{\partial ^{l}W_{j,k-1}(x)}{\partial x^{l}}\frac{%
\partial ^{i-l-1}y_{j}(x)}{\partial ^{i-l-1}x}=0.
\end{equation*}%
Another appeal to Leibniz rule yields%
\begin{gather*}
\frac{d^{l}W_{j,k-1}(x)}{dx^{l}}=\frac{d^{l}}{dx^{l}}\left[ \frac{W_{j}(%
\mathcal{F}_{k-1})(x)}{W(\mathcal{F}_{k-1})(x)}h_{k-1}(x)\right] \\
=\sum\limits_{r=0}^{l}c_{r}^{l}\frac{d^{r}}{dx^{r}}\left[ \frac{W_{j}(%
\mathcal{F}_{k-1})(x)}{W(\mathcal{F}_{k-1})(x)}\right] \frac{%
d^{l-r}h_{k-1}(x)}{d^{l-r}x},
\end{gather*}%
Therefore, the desired $h_{k-1}$ is the solution of the HODE%
\begin{gather*}
\sum\limits_{l=0}^{m-1}\sum\limits_{i=l+1}^{m}p_{i}(x)c_{l+1}^{i}\sum%
\limits_{j=1}^{k-1}\frac{\partial ^{i-l-1}y_{j}(x)}{\partial ^{i-l-1}x}%
\sum\limits_{r=0}^{l} \\
c_{r}^{l}\frac{d^{r}}{dx^{r}}\left[ \frac{W_{j}(\mathcal{F}_{k-1})(x)}{W(%
\mathcal{F}_{k-1})(x)}\right] \frac{d^{l-r}h_{k-1}(x)}{d^{l-r}x}=0,
\end{gather*}%
which can be written after changing the order of summation for $l$ and $r$ as%
\begin{gather*}
\sum\limits_{r=0}^{m-1}\sum\limits_{l=r}^{m-1}\sum%
\limits_{i=l+1}^{m}p_{i}(x)c_{l+1}^{i}\sum\limits_{j=1}^{k-1}\frac{\partial
^{i-l-1}y_{j}(x)}{\partial ^{i-l-1}x} \\
c_{r}^{l}\frac{d^{r}}{dx^{r}}\left[ \frac{W_{j}(\mathcal{F}_{k-1})(x)}{W(%
\mathcal{F}_{k-1})(x)}\right] \frac{d^{l-r}h_{k-1}(x)}{d^{l-r}x}=0.
\end{gather*}%
Letting $u=l-r$, we have%
\begin{gather*}
\sum\limits_{r=0}^{m-1}\sum\limits_{u=0}^{m-1-r}c_{r}^{u+r}\sum%
\limits_{i=u+r+1}^{m}p_{i}(x)c_{u+r+1}^{i}\sum\limits_{j=1}^{k-1}\frac{%
\partial ^{i-u-r-1}y_{j}(x)}{\partial ^{i-u-r-1}x} \\
\frac{d^{r}}{dx^{r}}\left[ \frac{W_{j}(\mathcal{F}_{k-1})(x)}{W(\mathcal{F}%
_{k-1})(x)}\right] \frac{d^{u}h_{k-1}(x)}{d^{u}x}=0,
\end{gather*}%
and changing the order of summation for $r$ and $u$ we finally obtain (\ref%
{HkHODE}), which is an $(m-1)$-order, linear HODE\ in $h_{k-1}.$ Note that a
sufficient condition for the solution $h_{k-1}$ to exist, is that the
coefficients (functions) of (\ref{HkHODE}) satisfy (\ref{CondhHODE}).
Unfortunately, as is the case with the reduction of method approach, a
closed form solution cannot be obtained immediately for $h_{k-1},$ $%
k=2,...,m,$ unless $m=2$ or we have specific conditions on the coefficients
of problem (\ref{LODE1}).

\bibliographystyle{plainnat}
\bibliography{WronskiProcess}

\end{document}